\documentclass[letterpaper,10 pt,conference]{ieeeconf}
\usepackage[utf8]{inputenc}
\usepackage{amsmath}
\usepackage{amsfonts}
\usepackage{amssymb}
\let\labelindent\relax \usepackage{enumitem}   
\usepackage{graphicx}   
\usepackage{color}
\usepackage{comment}
\usepackage{graphicx}
\usepackage{cite}
\usepackage{float}
\usepackage{pmat}
\newtheorem{theorem}{Theorem}

\newtheorem{proposition}{Proposition}

\newtheorem{definition}{Definition}

\DeclareMathOperator{\im}{im}

\newcommand{\bmu}{\mathbf{u}}
\newcommand{\bmx}{\mathbf{x}}

\newcommand{\bmw}{\mathbf{w}}
\newcommand{\bmy}{\mathbf{y}}

\IEEEoverridecommandlockouts                              

\title{\LARGE \bf
A Matrix Finsler's Lemma with Applications to Data-Driven Control
}

\author{Henk J. van Waarde and M. Kanat Camlibel
\thanks{Henk van Waarde is with the Department of Engineering, University of Cambridge, Trumpington Street, Cambridge CB2 1PZ, UK. Kanat Camlibel is with the Bernoulli Institute for Mathematics, Computer Science, and Artificial Intelligence, University of Groningen, Nij\-enborgh 9, 9747 AG, Groningen, The Netherlands. Email: {\tt\small hv280@cam.ac.uk,m.k.camlibel@rug.nl}.
}%
}

\begin{document}

\maketitle
\thispagestyle{empty}
\pagestyle{empty}

\begin{abstract}
In a recent paper it was shown how a matrix S-lemma can be applied to construct controllers from noisy data. The current paper complements these results by proving a matrix version of the classical Finsler's lemma. This matrix Finsler's lemma provides a tractable condition under which all matrix solutions to a quadratic equality also satisfy a quadratic inequality. We will apply this result to bridge known data-driven control design techniques for both exact and noisy data, thereby revealing a more general theory. The result is also applied to data-driven control of Lur'e systems. 
\end{abstract}

\section{Introduction}

Data-driven control refers to all approaches that use measured data as starting point in the control design. This design can be done either indirectly via model identification, or by directly mapping data to control policies. Both paradigms have a long history, but data-driven control has recently witnessed a renewed surge of interest, partly because of the widespread availability of data and the successes of machine learning algorithms. We mention contributions to data-driven optimal control \cite{Dean2019,Umenberger2019,Baggio2019,Coppens2020}, predictive control \cite{Coulson2020,Berberich2020c,Hewing2020,Yin2020b} and robust tracking control \cite{Xu2021}, nonlinear control \cite{Tabuada2020,Bisoffi2020,Alsalti2021} and system level synthesis \cite{Xue2020,Lian2021}.

Several recent papers aim at deriving tractable data-based linear matrix inequalities (LMI's) that enable direct data-driven control design. The paper \cite{Dai2018} proposes a semidefinite programming relaxation for the stabilization of switched systems. The authors of \cite{DePersis2020} provide a data-based parameterization of controllers, which is applied to stabilization and optimal control problems. In \cite{vanWaarde2020}, notions of informative data are defined, which leads to necessary and sufficient data-based conditions for different analysis and control problems. The paper \cite{Berberich2020} considers a noise bound in terms of a quadratic matrix inequality and proposes LMI conditions for control with guaranteed stability and performance. Combining data with prior knowledge on the system dynamics has been studied in \cite{Berberich2020d}. Also the problem of data-based verification of dissipativity properties has been cast as an LMI problem in \cite{Koch2020}.

An important question in this line of work regards the conservatism of the proposed LMI conditions. In this direction, a state-of-the-art result is the matrix generalization \cite{vanWaarde2020d} of the classical S-lemma \cite{Yakubovich1977}. This result provides an LMI condition under which all matrix solutions to one quadratic matrix inequality (QMI) also satisfy another QMI. The first inequality is motivated by the data: a quadratic bound on the noise, used in \cite{DePersis2020,Berberich2020,vanWaarde2020d}, has the consequence that all systems explaining the data satisfy a QMI. The second inequality captures design specifications such as stability or $\mathcal{H}_2$/$\mathcal{H}_{\infty}$ performance. Based on the matrix S-lemma, necessary and sufficient conditions could be provided for data-driven control with guaranteed quadratic stability and performance \cite{vanWaarde2020d}. 

A curious observation is that the stabilization result based on the S-lemma does not fully recover the stabilization result of \cite{vanWaarde2020} for \emph{noise-free} data. The reason is that in this case the Slater condition that is required for the matrix S-lemma does not hold. This fact is somewhat unsatisfactory because control using noise-free data should intuitively always be a special case of that for noisy data (with bound zero). 

In this paper we resolve this issue by introducing a matrix version of Finsler's lemma \cite{Finsler1936}. The classical Finsler's lemma provides an LMI condition under which a quadratic inequality is the consequence of a quadratic \emph{equality}. We will explain the difficulties in generalizing this result to matrix variables. Then, as our main contribution we will provide a Finsler's lemma for matrix variables in case the involved matrices obey some special structure. This matrix Finsler's lemma is then applied to data-driven stabilization. Interestingly, we will see that the LMI condition of \cite{vanWaarde2020d} is also necessary and sufficient in the special case of noise-free data, a result that could not be concluded from the matrix S-lemma \cite{vanWaarde2020d}. We believe that the matrix Finsler's lemma will also find other applications in situations where a QMI is the consequence of a matrix equality. In this paper, we will study one more of such situations, namely the construction of absolutely stabilizing controllers of Lur'e systems.

\emph{Outline}: In Section~\ref{sec:recap} we recap data-driven stabilization results and state the problem. Section~\ref{sec:Finsler} contains our results on the matrix Finsler's lemma. In Section~\ref{sec:bridge} this result is applied to bridge the results for noiseless \cite{vanWaarde2020} and noisy data \cite{vanWaarde2020d}. Finally, in Section~\ref{sec:Lure} we consider control of Lur'e systems. 
  
\section{Recap of data-driven stabilization and problem formulation}
\label{sec:recap}

We will first recap two data-driven stabilization results, for noise-free and noisy data, which can be found in the references \cite{vanWaarde2020,vanWaarde2020d}. Consider the system 
\begin{equation}
\label{sys}
\bmx(t+1) = A_s \bmx(t) + B_s \bmu(t) + \bmw(t),
\end{equation}
where $\bmx \in \mathbb{R}^n$ is the state, $\bmu \in \mathbb{R}^m$ is the control input and $\bmw \in \mathbb{R}^n$ denotes noise. The real matrices $A_s$ and $B_s$ are not assumed to be known. Instead of this, it is assumed that input/state data are obtained from \eqref{sys}, which are collected in the matrices
\begin{equation}
\label{XUmdata}
\begin{aligned}
X &= \begin{bmatrix}
x(0) & x(1) & \cdots & x(T) 
\end{bmatrix} \\
U_- &= \begin{bmatrix}
u(0) & u(1) & \cdots & u(T-1) 
\end{bmatrix}.
\end{aligned}
\end{equation}
We will also make use of shifted versions of the state sequence which are denoted by
\begin{equation}
\label{XpXm}
\begin{aligned}
X_- &= \begin{bmatrix}
x(0) & x(1) & \cdots & x(T-1)
\end{bmatrix} \\
X_+ &= \begin{bmatrix}
x(1) & x(2) & \cdots & x(T)
\end{bmatrix}.
\end{aligned}
\end{equation}

\subsection{Data-driven stabilization using exact data}
In this section we focus on the noise-free situation in which $\bmw = 0$. The purpose is to use the input/state data $(U_-,X)$ for the design of a stabilizing state feedback controller $\bmu = K \bmx$. Of course, this is only possible if the data contain sufficient information about the unknown system \eqref{sys}, i.e., if they are \emph{informative} for control design. 
\begin{definition}
\label{definfexact}
Suppose that the data $(U_-,X)$ are generated by \eqref{sys} for $\bmw = 0$. Then $(U_-,X)$ are \emph{informative for stabilization by state feedback} \cite{vanWaarde2020} if there exists a $K$ such that $A+BK$ is Schur stable for all $(A,B) \in \Sigma$, where
\begin{equation}
\label{Sigmaexact}
\Sigma := \{(A,B) \mid X_+ = AX_- + BU_- \}. 
\end{equation}
\end{definition}

\vspace{5pt}

The data are thus informative if there exists a single controller $K$ that stabilizes all systems explaining the data, i.e., all systems in $\Sigma$. 

Informativity for stabilization can be checked by solving a data-based linear matrix inequality (LMI), given in \eqref{LMIexact}. This LMI condition was proposed in \cite{DePersis2020}, and in \cite{vanWaarde2020} it was shown that it is \emph{necessary and sufficient} for informativity for stabilization. We state the result as follows.

\begin{proposition}
\label{propstabexact}
The data $(U_-,X)$ are informative for stabilization by state feedback if and only if there exists a matrix $\Theta \in \mathbb{R}^{T\times n}$ such that $X_- \Theta = (X_- \Theta)^\top$ and 
\begin{equation}
\label{LMIexact}
\begin{bmatrix}
X_- \Theta & X_+ \Theta \\ 
\Theta^\top X_+^\top & X_- \Theta 
\end{bmatrix} > 0. 
\end{equation}
Moreover, $K$ is such that $A+BK$ is stable for all $(A,B) \in \Sigma$ if and only if $K = U_-\Theta (X_- \Theta)^{-1}$ for a $\Theta$ satisfying \eqref{LMIexact}. 
\end{proposition}

\subsection{Data-driven stabilization using noisy data}

Next, we consider the system \eqref{sys} where $\bmw$ is not necessarily zero. The experimental input/state data are denoted by $(U_-,X)$, as before. This time, we also denote the noise samples during an experiment by 
$$
W_- = \begin{bmatrix}
w(0) & w(1) & \cdots & w(T-1)
\end{bmatrix}.
$$ 
Of course, the matrix $W_-$ is not known, but is assumed to bounded as 
\begin{equation}
    \label{asnoise}
    \begin{bmatrix}
    I \\ W_-^\top 
    \end{bmatrix}^\top 
    \begin{bmatrix}
    \Phi_{11} & \Phi_{12} \\
    \Phi_{12}^\top & \Phi_{22}
    \end{bmatrix}
    \begin{bmatrix}
    I \\ W_-^\top 
    \end{bmatrix} \geq 0,
\end{equation}
for known $\Phi_{11} = \Phi_{11}^\top$, $\Phi_{12}$ and $\Phi_{22} = \Phi_{22}^\top < 0$. This noise model was first introduced in \cite{vanWaarde2020d}. It can be interpreted as the transposed (or dual) model as the one used in \cite{Berberich2020}. The inequality \eqref{asnoise} has the interpretation that the energy of $\bmw$ is bounded on the finite time interval $[0,T-1]$.

Given the noise model \eqref{asnoise}, the set of all systems explaining the data is given by all $(A,B)$ such that
\begin{equation}
\label{eqXp}
X_+ = AX_-+BU_-+W_-
\end{equation}
is satisfied for some realization $W_-$ of the noise, that is,
$$
\Sigma_\Phi := \{(A,B) \mid \eqref{eqXp} \text{ holds for some } W_- \text{ satisfying } \eqref{asnoise} \}.
$$

With this in mind, we recall the following notion of informative data for stabilization using noisy data \cite{vanWaarde2020d}. 
\begin{definition}
\label{definformativity}
Suppose that the data $(U_-,X)$ are generated by \eqref{sys} for some noise sequence $W_-$ satisfying \eqref{asnoise}. Then $(U_-,X)$ are called \emph{informative for quadratic stabilization} if there exists a feedback gain $K$ and a matrix $P = P^\top > 0$ such that
\begin{equation}
    \label{lyapunovineq}
P - (A+BK) P (A+BK)^\top > 0
\end{equation}
for all $(A,B) \in \Sigma_\Phi$.
\end{definition}

Note that we focus on stabilization with a \emph{common} Lyapunov matrix $P$. 

A necessary and sufficient condition for informativity for quadratic stabilization was given in \cite{vanWaarde2020d}. The main concept that was used in that paper was a matrix version of the classical S-lemma. In fact, \cite{vanWaarde2020d} introduced both a non-strict and a strict version of this matrix S-lemma, both of which we recall in the following propositions. 

\begin{figure*}[t!]
\normalsize
\begin{equation}
\begin{bmatrix}
    P-\beta I & 0 & 0 & 0 \\
    0 & -P & -L^\top & 0 \\
    0 & -L & 0 & L \\
    0 & 0 & L^\top & P
    \end{bmatrix} - \begin{bmatrix}
    I & X_+ \\ 0 & -X_- \\ 0 & -U_- \\ 0 & 0
    \end{bmatrix}
    \begin{bmatrix}
    \Phi_{11} & \Phi_{12} \\
    \Phi_{12}^\top & \Phi_{22}
    \end{bmatrix}
    \begin{bmatrix}
    I & X_+ \\ 0 & -X_- \\ 0 & -U_- \\ 0 & 0
    \end{bmatrix}^\top \geq 0. \label{LMIstab}\tag{FS}
\end{equation}
\hrulefill
\vspace*{4pt}
\end{figure*}

\begin{proposition}[Matrix S-lemma]
\label{matSlemma}
Consider the symmetric matrices $M,N \in \mathbb{R}^{(k+\ell) \times (k+\ell)}$ and assume that there exists some matrix $\bar{Z} \in \mathbb{R}^{\ell \times k}$ such that
\begin{equation}
    \label{matSlater}
\begin{bmatrix} I \\ \bar{Z} \end{bmatrix}^\top N \begin{bmatrix} I \\ \bar{Z} \end{bmatrix} > 0.
\end{equation}
Then we have that 
    $$
    \begin{bmatrix} I \\ Z \end{bmatrix}^\top M \begin{bmatrix} I \\ Z \end{bmatrix} \geq 0 \:\: \forall  Z \in \mathbb{R}^{\ell \times k} \text{ such that} \begin{bmatrix} I \\ Z \end{bmatrix}^\top N \begin{bmatrix} I \\ Z \end{bmatrix} \geq 0
    $$
if and only if there exists a scalar $\alpha \!\geq\! 0$ such that $M \!- \alpha N \!\geq\! 0$.
\end{proposition}

\vspace{5pt}

\begin{proposition}[Strict matrix S-lemma]
\label{strictmatSlemma2}
Consider symme\-tric matrices $M,N \in \mathbb{R}^{(k+\ell) \times (k+\ell)}$, partitioned as 
\begin{equation}
\label{partitionMN}
M = \begin{bmatrix}
M_{11} & M_{12} \\ M_{12}^\top & M_{22}
\end{bmatrix} \text{ and } N = \begin{bmatrix}
N_{11} & N_{12} \\ N_{12}^\top & N_{22}
\end{bmatrix}.
\end{equation}
Assume that $M_{22} \leq 0$, $N_{22} \leq 0$ and $\ker N_{22} \subseteq \ker N_{12}$. Suppose that there exists some matrix $\bar{Z} \in \mathbb{R}^{\ell \times k}$ satisfying \eqref{matSlater}. Then we have that
$$
    \begin{bmatrix} I \\ Z \end{bmatrix}^\top M \begin{bmatrix} I \\ Z \end{bmatrix} > 0 \:\: \forall  Z \in \mathbb{R}^{\ell \times k} \text{ such that} \begin{bmatrix} I \\ Z \end{bmatrix}^\top N \begin{bmatrix} I \\ Z \end{bmatrix} \geq 0
    $$
    if and only if there exist $\alpha \geq 0$ and $\beta > 0$ such that 
    $$
    M- \alpha N \geq \begin{bmatrix}
    \beta I & 0 \\ 0 & 0
    \end{bmatrix}.
    $$
\end{proposition}

\vspace{5pt}

Based on the strict matrix S-lemma, the following characterization of informativity for quadratic stabilization can be established \cite{vanWaarde2020d}. For this, we define $N$ as 
\begin{equation}
\label{Nstab}
N := \begin{pmat}[{.}]
    I & X_+ \cr 0 & -X_- \cr 0 & -U_- \cr
    \end{pmat}
    \begin{bmatrix}
    \Phi_{11} & \Phi_{12} \\
    \Phi_{12}^\top & \Phi_{22}
    \end{bmatrix}
    \begin{pmat}[{.}]
    I & X_+ \cr 0 & -X_- \cr 0 & -U_- \cr
    \end{pmat}^\top.
\end{equation}

\begin{proposition}
\label{propstabnoise}
Assume that Slater condition \eqref{matSlater} holds for $N$ in \eqref{Nstab} and some $\bar{Z} \in \mathbb{R}^{(n+m)\times n}$. Then the data $(U_-,X)$ are informative for quadratic stabilization if and only if there exists an $n\times n$ matrix $P = P^\top > 0$, an $L \in \mathbb{R}^{m \times n}$ and a scalar $\beta > 0$ satisfying \eqref{LMIstab}.

Moreover, if $P$ and $L$ satisfy \eqref{LMIstab} then $K := L P^{-1}$ is a stabilizing feedback gain for all $(A,B) \in \Sigma_\Phi$.
\end{proposition}

We note that the original formulation in \cite{vanWaarde2020d} involved an additional scalar variable $\alpha$. However, in the stabilization problem in Proposition~\ref{propstabnoise} this variable can be absorbed in $P,L$ and $\beta$. 

\subsection{Problem formulation}

To summarize, in the case of noise-free data, Proposition~\ref{propstabexact} gives a necessary and sufficient condition for informativity for stabilization. Moreover, in the case of noisy data, Proposition~\ref{propstabnoise} provides a necessary and sufficient condition for informativity for quadratic stabilization. 

A natural question is now the following: what is the relation between these two propositions, and can the former be obtained as a special case from the latter?  

Surprisingly, the answer to this question is far from trivial. To initiate our investigation, it is tempting to consider the noise model \eqref{asnoise} with 
\begin{equation}
\label{Phiexact}
\Phi = \begin{bmatrix}
0 & 0 \\ 0 & -I
\end{bmatrix}.
\end{equation}
Indeed, this noise model implies that $W_- W_-^\top \leq 0$, i.e., $W_- = 0$ which corresponds exactly to the case in which the data are noise-free. 

Now, a problem arises when applying Proposition~\ref{propstabnoise} to noise models of the form \eqref{Phiexact}. The reason is that for $\Phi$ in \eqref{Phiexact}, the matrix $N$ in \eqref{Nstab} is negative semidefinite. In turn, this implies that the Slater condition \eqref{matSlater} is \emph{not satisfied}. The conclusion is that Proposition~\ref{propstabnoise} does not yield a necessary and sufficient condition for quadratic stabilization in the noise-free case (note that \emph{sufficiency} of \eqref{LMIstab} does hold, regardless of the Slater condition). 

Despite this potential shortcoming of Proposition~\ref{propstabnoise}, it turns out to be possible to bridge the results for exact and noisy data in Propositions~\ref{propstabexact} and \ref{propstabnoise}. In order to understand this relation we need a new result, namely a matrix version of \emph{Finsler's lemma}.

\section{The matrix Finsler's lemma}
\label{sec:Finsler}

Essentially, informativity for stabilization (Definition~\ref{definfexact}) asks for the existence of $P$ and $K$ such that a quadratic \emph{inequality} \eqref{lyapunovineq} holds for all  $(A,B)$ satisfying the \emph{equality} defined by \eqref{Sigmaexact}. This is more than reminiscent of the classical Finsler's lemma, named after Paul Finsler who proved the result in 1936. Two versions of Finsler's lemma are known, for both strict and non-strict inequalities. We will recall both results in the following two propositions that can be found in \cite{Finsler1936,Zi-zong2010}.

\begin{proposition}[Strict Finsler's lemma]
\label{strictFinsler}
Let $M,N \!\in \mathbb{R}^{\ell \times \ell}$ be symmetric. Then $x^\top M x > 0$ for all nonzero $x\in \mathbb{R}^\ell$ satisfying $x^\top N x = 0$ if and only if there exists an $\alpha \in \mathbb{R}$ such that $M - \alpha N > 0$. 
\end{proposition}

\begin{proposition}[Non-strict Finsler's lemma]
\label{nonstrictFinsler}
Let $M,N \in \mathbb{R}^{\ell \times \ell}$ be symmetric and assume that $N$ is indefinite. Then $x^\top M x \geq 0$ for all $x\in \mathbb{R}^\ell$ satisfying $x^\top N x = 0$ if and only if there exists an $\alpha \in \mathbb{R}$ such that $M - \alpha N \geq 0$. 
\end{proposition}

In the spirit of Finsler's lemma, we would like to find a tractable characterization of a statement of the form
$$
\begin{bmatrix} I \\ Z \end{bmatrix}^\top \! M \! \begin{bmatrix} I \\ Z \end{bmatrix} \geq 0 \:\: \forall  Z \in \mathbb{R}^{\ell \times k} \text{ such that} \begin{bmatrix} I \\ Z \end{bmatrix}^\top \! N \! \begin{bmatrix} I \\ Z \end{bmatrix} = 0,
$$
where the inequality involving $M$ is either non-strict or strict. Note that, in contrast to Finsler's lemma, this statement involves \emph{inhomogeneous} functions of \emph{matrix variables}.

To motivate our main result, we first point out some difficulties that arise when attempting to generalize Propositions~\ref{strictFinsler} and \ref{nonstrictFinsler} to the matrix case. First, the strict Finsler's lemma does not directly generalize to inhomogeneous quadratic functions, even in the vector-valued case. To convince oneself of this fact, it is sufficient to realize that 
$$
\begin{bmatrix}
1 \\ z
\end{bmatrix}^\top \!\! \begin{bmatrix}
1 & 0 \\ 0 & 0
\end{bmatrix} \! \begin{bmatrix}
1 \\ z
\end{bmatrix} \! > \! 0 \:\: \forall z \in \mathbb{R} \text{ such that } \!
\begin{bmatrix}
1 \\ z
\end{bmatrix}^\top \!\! \begin{bmatrix}
0 & 1 \\ 1 & 0
\end{bmatrix} \! \begin{bmatrix}
1 \\ z
\end{bmatrix} = 0,
$$
while the matrix 
$$
\begin{bmatrix}
1 & -\alpha \\ - \alpha & 0
\end{bmatrix}
$$
clearly cannot be positive definite. Secondly, in the non-strict case, we note that Proposition~\ref{nonstrictFinsler} requires a Slater condition as $N$ is assumed to be indefinite. This Slater condition is problematic for data-driven control since we already know that in the noise-free setting the matrix $N$ in \eqref{Nstab} is negative semidefinite. A generalization of Proposition~\ref{nonstrictFinsler}, even if possible, thus appears to be of lesser interest. 

Our solution to this is the following: we will develop a matrix Finsler's lemma for matrices $M$ and $N$ with \emph{specific structure}, and without assuming any type of Slater condition. Our main result can be formulated as follows. We will use $X^+$ to denote the Moore-Penrose pseudo-inverse of $X$.

\begin{theorem}[Matrix Finsler's lemma]
\label{t:matFinsler}
Consider symme\-tric matrices $M,N \in \mathbb{R}^{(k+\ell) \times (k+\ell)}$ partitioned as in \eqref{partitionMN}. Assume that 
\begin{enumerate}
\item $M_{12} = 0$ and $M_{22} \leq 0$. \label{as1}
\item $N_{22} \leq 0$ and $N_{11} - N_{12} N_{22}^+ N_{12}^\top = 0$. \label{as2}
\item $\exists \: G$ such that $M_{11} + G^\top M_{22} G > 0$ and $N_{22} G = N_{12}^\top$. \label{as3}
\end{enumerate}
Then we have that 
\vspace{-2pt}
\begin{equation}
\label{implicationFinsler}
\begin{bmatrix} I \\ Z \end{bmatrix}^\top \! M \! \begin{bmatrix} I \\ Z \end{bmatrix} \geq 0 \:\: \forall  Z \in \mathbb{R}^{\ell \times k} \text{ such that} \begin{bmatrix} I \\ Z \end{bmatrix}^\top \! N \! \begin{bmatrix} I \\ Z \end{bmatrix} = 0
\vspace{-2pt}
\end{equation}
if and only if there exists $\alpha \in \mathbb{R}$ such that $M - \alpha N \geq 0$.
\end{theorem}

\begin{proof}
The ``if" part is obvious. We thus focus on proving the ``only if" part. By Assumption \ref{as2}, the matrix $\bar{Z} := -N_{22}^+N_{12}^\top$ satisfies 
\vspace{-2pt}
$$
\begin{bmatrix}
I \\ \bar{Z}
\end{bmatrix}^\top N 
\begin{bmatrix}
I \\ \bar{Z}
\end{bmatrix} = 0.
\vspace{-2pt}
$$
Now, let $\hat{Z} := \xi \eta^\top$ where $\xi \in \ker N_{22}$ and $\eta$ is a nonzero vector. By hypothesis, we have 
\vspace{-2pt}
\begin{equation}
\label{Zhat}
\begin{bmatrix}
I \\ \bar{Z} + \gamma \hat{Z}
\end{bmatrix}^\top M \begin{bmatrix}
I \\ \bar{Z} + \gamma \hat{Z}
\end{bmatrix} \geq 0
\vspace{-2pt}
\end{equation}
for all $\gamma \in \mathbb{R}$. Recall that by Assumption \ref{as1}, $M_{22} \leq 0$. This implies that $M_{22} \hat{Z} = 0$, for otherwise there exists a sufficiently large $\gamma \in \mathbb{R}$ violating \eqref{Zhat}. We have thus proven that $\ker N_{22} \subseteq \ker M_{22}$, equivalently, $\im M_{22} \subseteq \im N_{22}$. 

Next, define the matrix 
\vspace{-2pt}
$$
T := \begin{bmatrix}
I & 0 \\ -N_{22}^+ N_{12}^\top & I
\end{bmatrix},
\vspace{-2pt}
$$
and compute 
\vspace{-2pt}
$$
T^\top N T = \begin{bmatrix}
0 & 0 \\ 0 & N_{22}
\end{bmatrix}
\vspace{-2pt}
$$
and 
\vspace{-2pt}
\begin{align*}
T^\top M T &= \begin{bmatrix}
M_{11} + N_{12} N_{22}^+ M_{22}N_{22}^+ N_{12}^\top & -N_{12} N_{22}^+ M_{22} \\
-M_{22} N_{22}^+ N_{12}^\top & M_{22}
\end{bmatrix} \\
&= \begin{bmatrix}
M_{11} + G^\top M_{22} G & -G^\top M_{22} \\
-M_{22} G & M_{22}
\end{bmatrix},
\vspace{-2pt}
\end{align*}
where for the last equality we have used Assumption \ref{as3} as follows: since $N_{12} = G^\top N_{22}$ and $\im M_{22} \subseteq \im N_{22}$ we have $N_{12} N_{22}^+ M_{22} = G^\top M_{22}$. Similarly, we conclude that $N_{12} N_{22}^+ M_{22} N_{22}^+N_{12}^\top = G^\top M_{22} G$. These computations reveal that
\vspace{-2pt}
$$
T^\top (M-\alpha N) T = \begin{bmatrix}
M_{11} + G^\top M_{22} G & -G^\top M_{22} \\
-M_{22} G & M_{22} - \alpha N_{22}
\end{bmatrix}.
\vspace{-2pt}
$$
Finally, by Assumption \ref{as3}, $M_{11} + G^\top M_{22} G > 0$ and thus it holds that $T^\top (M-\alpha N) T \geq 0$ if and only if 
\vspace{-2pt}
\begin{equation}
\label{ineqM22N22}
M_{22} - \alpha N_{22} - M_{22} G (M_{11} + G^\top M_{22} G)^{-1} G^\top M_{22} \geq 0. 
\vspace{-2pt}
\end{equation}

By Assumption \ref{as2}, $N_{22} \leq 0$ and since $\ker N_{22} \subseteq \ker M_{22}$, we conclude that there exists a sufficiently large $\alpha \in \mathbb{R}$ such that \eqref{ineqM22N22} holds. This implies that there exists an $\alpha \in \mathbb{R}$ such that $M-\alpha N \geq 0$, proving the theorem.
\end{proof}

\section{Bridging the exact and noisy cases}
\label{sec:bridge}

In this section, we will apply the matrix Finsler's lemma to find a new characterization of informativity for stabilization in the exact data case, thereby bridging the exact and noisy formulations. The result can be formulated as follows. 

\begin{theorem}
\label{t:stabexact}
Let the data $(U_-,X)$ be generated by \eqref{sys} with $\bmw = 0$. Then $(U_-,X)$ are informative for stabilization by state feedback if and only if there exist $P = P^\top > 0$ and $L$, and a scalar $\beta > 0$ satisfying
\vspace{-2pt}
\begin{equation}
\begin{bmatrix}
    P-\beta I & 0 & 0 & 0 \\
    0 & -P & -L^\top & 0 \\
    0 & -L & 0 & L \\
    0 & 0 & L^\top & P
    \end{bmatrix} + \begin{bmatrix}
    X_+ \\ -X_- \\ -U_- \\ 0
    \end{bmatrix}\!\!
    \begin{bmatrix}
    X_+ \\ -X_- \\ -U_- \\ 0
    \end{bmatrix}^\top \geq  0. \label{LMIstab2}
    \vspace{-2pt}
\end{equation}
Moreover, if $P$ and $L$ satisfy \eqref{LMIstab2} then $K := LP^{-1}$ is a stabilizing feedback gain for all $(A,B) \in \Sigma$. 
\end{theorem}

\begin{proof}
To prove the ``if" part, suppose that \eqref{LMIstab2} is feasible and define $K := LP^{-1}$. Compute the Schur complement of \eqref{LMIstab2} with respect to the fourth row and column block, which yields
\vspace{-2pt}
\begin{equation}
\label{intermediate}
\begin{bmatrix}
    P-\beta I & 0 & 0 \\
    0 & -P & -PK^\top \\
    0 & -KP & -KPK \\
    \end{bmatrix}
   + \begin{bmatrix}
    X_+ \\ -X_- \\ -U_- 
    \end{bmatrix}\!\!
    \begin{bmatrix}
    X_+ \\ -X_- \\ -U_- 
    \end{bmatrix}^\top \geq  0.
    \vspace{-2pt}
\end{equation}
Finally, for any $(A,B) \in \Sigma$, multiply \eqref{intermediate} from the left by $\begin{bmatrix}
I & A & B
\end{bmatrix}$ and from right by its transposed. This results in
\vspace{-2pt}
\begin{equation}
\label{lyapproof}
P - (A+BK)P(A+BK)^\top \geq \beta I > 0,
\vspace{-2pt}
\end{equation}
proving that $A+BK$ is Schur stable. Thus, the data $(U_-,X)$ are informative for stabilization by state feedback and $K = LP^{-1}$ is a stabilizing controller for all $(A,B) \in \Sigma$. 

Next, to prove the ``only if" part, suppose that the data $(U_-,X)$ are informative for stabilization by state feedback. Then there exists a controller $K$ such that $A+BK$ is Schur for all $(A,B) \in \Sigma$. By \cite[Lem. 15]{vanWaarde2020} there exist $P = P^\top > 0$ and $\beta > 0$ such that \eqref{lyapproof} holds for all $(A,B) \in \Sigma$. Define the partitioned matrices 
\vspace{-2pt}
\begin{align*}
    M &= \begin{pmat}[{|}]
M_{11} & M_{12} \cr\-
M_{12}^\top & M_{22} \cr
\end{pmat} := \begin{pmat}[{|.}]
    P-\beta I & 0 & 0 \cr\-
    0 & -P & -PK^\top \cr
    0 & -KP & -KPK^\top \cr
    \end{pmat} \\
N &= \begin{pmat}[{|}]
N_{11} & N_{12} \cr\- N_{12}^\top & N_{22} \cr
\end{pmat} := -\begin{pmat}[{.}]
    X_+ \cr\- -X_- \cr -U_- \cr
    \end{pmat}
    \begin{pmat}[{.}]
    X_+ \cr\- -X_- \cr -U_- \cr
    \end{pmat}^\top. 
    \vspace{-2pt}
\end{align*}
For these matrices, the statement in \eqref{implicationFinsler} holds. In addition, note that $M_{12} = 0$ and $M_{22} \leq 0$, thus Assumption 1) of Theorem~\ref{t:matFinsler} holds. Similarly, $N_{22} \leq 0$, and $N_{11} - N_{12} N_{22}^+ N_{12}^\top$ equals
\vspace{-2pt}
\begin{align*}
X_+ \left(-I + \begin{bmatrix}
X_- \\ U_-
\end{bmatrix}^\top \left( \begin{bmatrix}
X_- \\ U_-
\end{bmatrix}\begin{bmatrix}
X_- \\ U_-
\end{bmatrix}^\top \right)^+ \begin{bmatrix}
X_- \\ U_-
\end{bmatrix}\right) X_+^\top,
\vspace{-2pt}
\end{align*}
which is zero since $\im X_+^\top \subseteq \im \begin{bmatrix}
X_-^\top & U_-^\top 
\end{bmatrix}$ by hypothesis. Therefore, Assumption 2) of Theorem~\ref{t:matFinsler} is also satisfied. Finally, note that $G:= \begin{bmatrix}
A_s & B_s
\end{bmatrix}^\top$ satisfies Assumption 3). We conclude by Theorem~\ref{t:matFinsler} that there exists $\alpha \in \mathbb{R}$ such that $M-\alpha N \geq 0$. In fact, we necessarily have $\alpha > 0$ since $-P < 0$. We can thus assume without loss that $\alpha = 1$ (as $P$ and $\beta$ can be scaled by $1/\alpha$). Finally, by defining $L = KP$ and using a Schur complement argument we conclude that \eqref{LMIstab2} is feasible. 
\end{proof}

Theorem~\ref{t:stabexact} bridges the exact and noisy case in the following sense. Note that the matrix on the right of \eqref{LMIstab2} equals
\vspace{-2pt}
$$
- \begin{bmatrix}
    I & X_+ \\ 0 & -X_- \\ 0 & -U_- \\ 0 & 0
    \end{bmatrix}
    \begin{bmatrix}
    0 & 0 \\
    0 & -I
    \end{bmatrix}
    \begin{bmatrix}
    I & X_+ \\ 0 & -X_- \\ 0 & -U_- \\ 0 & 0
    \end{bmatrix}^\top,
\vspace{-2pt}
$$
which is nothing but a special case of the matrix on the right of \eqref{LMIstab} for the choices $\Phi_{11} = 0$, $\Phi_{12} = 0$ and $\Phi_{22} = -I$. This means that feasibility of the LMI \eqref{LMIstab} (with specific $\Phi$) is also necessary and sufficient for informativity for stabilization in the case of exact data. In this case, we even know that the assumption of a common Lyapunov function is not restrictive, i.e., informativity for stabilization is equivalent for informativity for quadratic stabilization if $\bmw = 0$. This follows directly from \cite[Lem. 15]{vanWaarde2020}.

Given the two equivalent conditions in Proposition~\ref{propstabexact} and Theorem~\ref{t:stabexact} it is natural to question the relative merits of both approaches. First of all, we note that the LMI conditions in \eqref{LMIexact} and \eqref{LMIstab2} are different in nature: the variable $\Theta$ in \eqref{LMIexact} has dimension $T \times n$ which depends on the time horizon of the experiment, while the dimensions of the variables $P,L$ and $\beta$ in \eqref{LMIstab2} are independent of $T$. From a computational point of view, Theorem~\ref{t:stabexact} may thus be preferred in cases where the inputs of the experiment are chosen to be \emph{persistently exciting} \cite{Willems2005,vanWaarde2020c} since this puts a lower bound $T \geq n+m+nm$ on the required number of samples. On the other hand, it has recently been shown \cite{vanWaarde2021} that for controllable pairs $(A_s,B_s)$, the data $(U_-,X)$ can be made informative for stabilization with at most $T = n+m$ samples, using an online input design method. In this case, the LMI \eqref{LMIexact} may be preferred since \eqref{LMIexact} has dimension $2n \times 2n$ which is smaller than the dimension $(3n+m) \times (3n+m)$ of \eqref{LMIstab2}.  

\section{Data-driven stabilization of Lur'e systems}
\label{sec:Lure}
In this section, we will apply the matrix Finsler's lemma to control Lur'e systems. First, we will explain the classical problem of absolute stability for such systems. Consider the Lur'e system
\vspace{-2pt}
\begin{equation}
\begin{aligned}
\label{Lure}
\bmx(t+1) = A\bmx(t) + B \bmu(t) + E\phi(C\bmx(t)) 
\end{aligned}
\vspace{-2pt}
\end{equation} 
where $\bmx \in \mathbb{R}^n$ is the state, $\bmu \in \mathbb{R}^m$ is the input and  $\phi : \mathbb{R} \to \mathbb{R}$ is a (nonlinear) function satisfying the sector condition
\vspace{-2pt}
\begin{equation}
\label{sector}
\phi(y) (\phi(y)-y) \leq 0 \quad \forall y \in \mathbb{R}.
\vspace{-2pt}
\end{equation}
The real matrices $A, B, E$ and $C$ are of appropriate dimensions. Suppose that we apply a state feedback controller $\bmu = K \bmx$ resulting in 
\vspace{-2pt}
\begin{equation}
\label{Lureclosed}
\bmx(t+1) = (A+BK)\bmx(t) + E\phi(C\bmx(t)). 
\vspace{-2pt}
\end{equation}
For systems of the form \eqref{Lureclosed}, a problem with a rich history is that of \emph{absolute stability}, i.e. global asymptotic stability of $0$ \emph{for all} sector-bounded nonlinearities, c.f. \cite{Boyd1994} for references. We focus on proving absolute stability of \eqref{Lureclosed} by means of a quadratic Lyapunov function $V(z) := z^\top P z$ where $P = P^\top > 0$. We thus want that $V(x(t+1)) < V(x(t))$ for all sector-bounded nonlinearities $\phi$ and all nonzero $x(t)$ and resulting $x(t+1)$ satisfying \eqref{Lureclosed}. We will mimic the continuous-time setting of \cite[Ch. 5]{Boyd1994}. Let $A_K := A+BK$. Then we require
\vspace{-2pt}
$$
 (A_Kx+Ew)^\top P (A_Kx+Ew) - x^\top P x < 0
 \vspace{-2pt}
$$
for all $w \in \mathbb{R}$ and nonzero $x \in \mathbb{R}^n$ satisfying $w (w-Cx) \leq 0$. Equivalently, 
\vspace{-2pt}
\begin{equation}
\label{ineq1}
\begin{bmatrix}
x \\ w
\end{bmatrix}^\top 
\begin{bmatrix}
P - A_K^\top P A_K & -A_K^\top P E \\ -E^\top P A_K & -E^\top P E 
\end{bmatrix}
\begin{bmatrix}
	x \\ w
\end{bmatrix} > 0
\vspace{-2pt}
\end{equation}
for all $w \in \mathbb{R}$ and nonzero $x \in \mathbb{R}^n$ satisfying
\vspace{-2pt}
\begin{equation}
\label{ineq2}
\begin{bmatrix}
x \\ w
\end{bmatrix}^\top 
\begin{bmatrix}
0 & \frac{1}{2}C^\top \\ \frac{1}{2} C & -1
\end{bmatrix}
\begin{bmatrix}
	x \\ w
\end{bmatrix} \geq 0.
\vspace{-2pt}
\end{equation}
Since \eqref{ineq2} is not satisfied when $x = 0$ and $w \neq 0$, the latter statement is equivalent to \eqref{ineq1} being satisfied for all nonzero $(x,w)$ satisfying \eqref{ineq2}. Assuming $C \neq 0$, the inequality \eqref{ineq2} is strictly feasible. Thus, by the S-lemma \cite[p. 24]{Boyd1994} we conclude that \eqref{ineq1} is satisfied for all nonzero $(x,w)$ satisfying \eqref{ineq2} if and only if 
\vspace{-2pt}
\begin{equation}
\label{ineqPalpha}
\begin{bmatrix}
P - A_K^\top P A_K & -A_K^\top P E \\ -E^\top P A_K & -E^\top P E 
\end{bmatrix} - \alpha \begin{bmatrix}
0 & \frac{1}{2}C^\top \\ \frac{1}{2} C & -1
\end{bmatrix} > 0
\vspace{-2pt}
\end{equation}
for some scalar $\alpha \geq 0$. Proving absolute stability of \eqref{Lureclosed} by a quadratic Lyapunov function thus boils down to finding $P = P^\top > 0$ and $\alpha \geq 0$ such that \eqref{ineqPalpha} holds. By homogeneity, we can even get rid of $\alpha$ and look for $P = P^\top > 0$ satisfying 
\vspace{-2pt}
\begin{equation}
\label{LMIabsstab}
\begin{bmatrix}
P - A_K^\top P A_K & -A_K^\top P E-\frac{1}{2}C^\top \\ -E^\top P A_K-\frac{1}{2}C & 1-E^\top P E 
\end{bmatrix} > 0.
\vspace{-2pt}
\end{equation} 

\subsection{Data-driven stabilization}
 Next, we consider the system  
 \vspace{-2pt}
\begin{equation}
\label{Lure2}
\bmx(t+1) = A_s\bmx(t) + B_s \bmu(t) + E_s \phi(C\bmx(t))
\vspace{-2pt}
\end{equation}
where $A_s, B_s$ and $E_s$ are unknown but the matrix $C$ is known\footnote{This assumption can be replaced by measurements of $\bmy(t) := C \bmx(t)$.}. We aim at constructing an absolutely stabilizing controller $\bmu = K\bmx$ on the basis of measurements $X$ and $U_-$ as in \eqref{XUmdata} and
\vspace{-2pt}
\begin{equation}
\label{Wdata}
\begin{aligned}
W_-\! &= \! \begin{bmatrix}
\phi(Cx(0)) \!\! & \phi(Cx(1)) & \!\!\! \cdots \!\!\! & \phi(Cx(T-1)) 
\end{bmatrix}.
\end{aligned}
\vspace{-2pt}
\end{equation}  
If we define $X_+$ and $X_-$ as in \eqref{XpXm} then all systems $(A,B,E)$ explaining the data are given by the set $\Sigma$ defined by 
\vspace{-2pt}
$$
\Sigma := \{(A,B,E) \mid X_+ = A X_- + B U_- + E W_- \}.
\vspace{-2pt}
$$
\begin{definition}
Suppose that the data $(U_-,W_-,X)$ in \eqref{XUmdata} and \eqref{Wdata} have been generated by \eqref{Lure2}. Then $(U_-,W_-,X)$ are called \emph{informative for absolute quadratic stabilization} if there exist  $P = P^\top > 0$ and $K$ such that \eqref{LMIabsstab} holds for all $(A,B,E) \in \Sigma$. 
\end{definition}

\begin{theorem}
The data $(U_-,W_-,X)$ are informative for absolute quadratic stabilization if and only if there exist matrices $Q = Q^\top > 0$ and $L$ and scalars $\alpha \in \mathbb{R}$ and $\beta > 0$ such that $CQC^\top < 4$ and
\vspace{-2pt}
\begin{equation}
\begingroup
\setlength\arraycolsep{1.8pt}
\begin{bmatrix}
    \!Q\!-\!\beta I & 0 & 0 & 0 & 0 &  0 \\
    0 & 0 & 0 & 0 & Q & 0 \\
    0 & 0 & 0 & 0 & L & 0 \\
    0 & 0 & 0 & 0 & 0 & 1 \\
    0 & Q & L^\top & 0 & Q & -\frac{1}{2}QC^\top\! \\
    0 & 0 & 0 & 1 & -\frac{1}{2}CQ & 1
    \end{bmatrix}
    \endgroup \!+\alpha\!\begin{bmatrix}
    X_+ \\ -X_- \\ -U_- \\ \!-W_- \\ 0 \\ 0
    \end{bmatrix}\!\!\!
    \begingroup
\setlength\arraycolsep{1.8pt}
    \begin{bmatrix}
    X_+ \\ -X_- \\ -U_- \\ \!-W_- \\ 0 \\ 0
    \end{bmatrix}^\top \endgroup \!\!\!\!\!\! \geq \! 0. \label{LMIabsstabdata}
    \vspace{-2pt}
\end{equation}
\newpage 
In this case, $K := LQ^{-1}$ is such that \eqref{Lureclosed} is absolutely stable for all $(A,B,E) \in \Sigma$. 
\end{theorem}

\begin{proof}
We first prove the ``if" part. Note that the lower right $2$ by $2$ block matrix of \eqref{LMIabsstabdata} is positive definite since $Q > 0$ and $CQC^\top < 4$. Define $P := Q^{-1}$ and $K:=LQ^{-1}$, and let $(A,B,E) \in \Sigma$. Multiply \eqref{LMIabsstabdata} from both sides by the block diagonal matrix with blocks $I, I,1, 1, P$ and $1$. Then take the Schur complement of \eqref{LMIabsstabdata} with respect to the lower right $2$ by $2$ block, and multiply with $\begin{bmatrix} I & A & B & E \end{bmatrix}$ from left and its transposed from right to obtain
\vspace{-2pt}
\begin{equation}
\label{strictineq}
\begingroup
\setlength\arraycolsep{1.8pt}
\begin{bmatrix}
I \\ A^\top \\ B^\top \\ E^\top
\end{bmatrix}^{\!\!\top} \!\!\!
\begin{bmatrix}
    P^{-1} & 0 \\
    0 & \!-\!\!\begin{bmatrix}
    I & 0 \\ K & 0 \\ 0 & I
    \end{bmatrix} \!\!\!
    \begin{bmatrix}
    P & -\frac{1}{2}C^\top \\
    -\frac{1}{2}C & 1
    \end{bmatrix}^{-1} \!\!
    \begin{bmatrix}
    I & 0 \\ K & 0 \\ 0 & I
    \end{bmatrix}^{\!\!\top}\! 
    \end{bmatrix} \!\!\!
    \begin{bmatrix}
I \\ A^\top \\ B^\top \\ E^\top
\end{bmatrix}
     \!>\! 0
     \endgroup
     \vspace{-2pt}
\end{equation}
Finally, by using a Schur complement argument twice, we see that \eqref{strictineq} implies \eqref{LMIabsstab}. Therefore the data are informative for absolute quadratic stabilization and $K$ is a suitable controller with Lyapunov matrix $Q^{-1}$. 

To prove the ``only if" part, suppose that there exist $P = P^\top > 0$ and $K$ such that \eqref{LMIabsstab} holds for all $(A,B,E) \in \Sigma$. Using a Schur complement argument twice this implies \eqref{strictineq} holds. Analogous to \cite[Lem. 15]{vanWaarde2020} it can be shown that $A+BK$ and $E$ are the same for all $(A,B,E) \in \Sigma$. This implies that \eqref{strictineq} still holds for all $(A,B,E) \in \Sigma$ if we replace the strict inequality by a non-strict inequality and $P^{-1}$ by $P^{-1} - \beta I$ for some sufficienctly small $\beta > 0$. Define
\vspace{-2pt} 
\begin{align*}
M &:= \! \begin{pmat}[{|}]
    \!P^{-1}\!-\!\beta I\! & 0 \cr\-
    0 & {-\!\!\begin{bmatrix}
    I & 0 \\ K & 0 \\ 0 & I
    \end{bmatrix}\!\!
    \begin{bmatrix}
    P & -\frac{1}{2}C^\top \\
    -\frac{1}{2}C & 1
    \end{bmatrix}^{-1} \!\!
    \begin{bmatrix}
    I & 0 \\ K & 0 \\ 0 & I
    \end{bmatrix}^{\!\!\top}} \cr
    \end{pmat} \\
    \vspace{-2pt}
    N &:= -\begin{pmat}[{.}]
    X_+ \cr\- -X_- \cr -U_- \cr -W_- \cr
    \end{pmat} \!\!
    \begin{pmat}[{.}]
    X_+ \cr\- -X_- \cr -U_- \cr -W_- \cr
    \end{pmat}^{\!\top}.
    \vspace{-4pt}
\end{align*}
By Theorem~\ref{t:matFinsler}, we conclude that there exists an $\alpha \in \mathbb{R}$ such that $M-\alpha N \geq 0$. Finally, by defining the variables $Q := P^{-1}$ and $L := KQ$ and using a Schur complement argument, we see that $CQC^\top < 4$ and \eqref{LMIabsstabdata} is feasible.
\end{proof}

\vspace{-4pt}

\section*{Acknowledgement}
We would like to thank Bart Besselink for discussions on the topic of Section~\ref{sec:Lure} of this paper.

\vspace{-4pt}

\bibliographystyle{IEEEtran}
\bibliography{references}

\begin{thebibliography}{10}
\providecommand{\url}[1]{#1}
\csname url@samestyle\endcsname
\providecommand{\newblock}{\relax}
\providecommand{\bibinfo}[2]{#2}
\providecommand{\BIBentrySTDinterwordspacing}{\spaceskip=0pt\relax}
\providecommand{\BIBentryALTinterwordstretchfactor}{4}
\providecommand{\BIBentryALTinterwordspacing}{\spaceskip=\fontdimen2\font plus
\BIBentryALTinterwordstretchfactor\fontdimen3\font minus
  \fontdimen4\font\relax}
\providecommand{\BIBforeignlanguage}[2]{{%
\expandafter\ifx\csname l@#1\endcsname\relax
\typeout{** WARNING: IEEEtran.bst: No hyphenation pattern has been}%
\typeout{** loaded for the language `#1'. Using the pattern for}%
\typeout{** the default language instead.}%
\else
\language=\csname l@#1\endcsname
\fi
#2}}
\providecommand{\BIBdecl}{\relax}
\BIBdecl

\bibitem{Dean2019}
S.~Dean, H.~Mania, N.~Matni, B.~Recht, and S.~Tu, ``On the sample complexity of
  the linear quadratic regulator,'' \emph{Foundations of Computational
  Mathematics}, Aug 2019.

\bibitem{Umenberger2019}
J.~Umenberger, M.~Ferizbegovic, T.~B. Sch\"{o}n, and H.~Hjalmarsson, ``Robust
  exploration in linear quadratic reinforcement learning,'' in \emph{Advances
  in Neural Information Processing Systems 32}.\hskip 1em plus 0.5em minus
  0.4em\relax Curran Associates, Inc., 2019, pp. 15\,336--15\,346.

\bibitem{Baggio2019}
G.~{Baggio}, V.~{Katewa}, and F.~{Pasqualetti}, ``Data-driven minimum-energy
  controls for linear systems,'' \emph{IEEE Control Systems Letters}, vol.~3,
  no.~3, pp. 589--594, July 2019.

\bibitem{Coppens2020}
P.~Coppens, M.~Schuurmans, and P.~Patrinos, ``Data-driven distributionally
  robust {LQR} with multiplicative noise,'' in \emph{Proc. of the Conference on
  Learning for Dynamics and Control}, Jun 2020, pp. 521--530.

\bibitem{Coulson2020}
J.~{Coulson}, J.~{Lygeros}, and F.~{D\"orfler}, ``Distributionally robust
  chance constrained data-enabled predictive control,''
  \emph{arxiv.org/abs/2006.01702}, 2020.

\bibitem{Berberich2020c}
J.~{Berberich}, J.~{Koehler}, M.~A. {Muller}, and F.~{Allg\"ower},
  ``Data-driven model predictive control with stability and robustness
  guarantees,'' \emph{IEEE Transactions on Automatic Control}, pp. 1--1, 2020.

\bibitem{Hewing2020}
L.~Hewing, K.~P. Wabersich, M.~Menner, and M.~N. Zeilinger, ``Learning-based
  model predictive control: Toward safe learning in control,'' \emph{Annual
  Review of Control, Robotics, and Autonomous Systems}, vol.~3, no.~1, pp.
  269--296, 2020.

\bibitem{Yin2020b}
M.~Yin, A.~Iannelli, and R.~S. Smith, ``Maximum likelihood signal matrix model
  for data-driven predictive control,'' \emph{arxiv.org/abs/2012.04678}, 2020.

\bibitem{Xu2021}
L.~Xu, M.~S. Turan, B.~Guo, and G.~{Ferrari-Trecate}, ``A data-driven convex
  programming approach to worst-case robust tracking controller design,''
  \emph{arxiv.org/pdf/2102.11918}, 2021.

\bibitem{Tabuada2020}
P.~Tabuada and L.~Fraile, ``Data-driven stabilization of {SISO} feedback
  linearizable systems,'' \emph{arxiv.org/abs/2003.14240}, 2020.

\bibitem{Bisoffi2020}
A.~Bisoffi, C.~{De Persis}, and P.~Tesi, ``Data-based stabilization of unknown
  bilinear systems with guaranteed basin of attraction,'' \emph{Systems \&
  Control Letters}, vol. 145, p. 104788, 2020.

\bibitem{Alsalti2021}
M.~Alsalti, J.~Berberich, V.~G. Lopez, F.~Allg\"ower, and M.~A. Müller,
  ``Data-based system analysis and control of flat nonlinear systems,''
  \emph{arxiv.org/abs/2103.02892}, 2021.

\bibitem{Xue2020}
A.~Xue and N.~Matni, ``Data-driven system level synthesis,''
  \emph{arxiv.org/abs/2011.10674}, 2020.

\bibitem{Lian2021}
Y.~Lian and C.~N. Jones, ``From system level synthesis to robust closed-loop
  data-enabled predictive control,'' \emph{arxiv.org/abs/2102.06553}, 2021.

\bibitem{Dai2018}
T.~{Dai} and M.~{Sznaier}, ``A moments based approach to designing {MIMO} data
  driven controllers for switched systems,'' in \emph{Proc. of the IEEE
  Conference on Decision and Control}, 2018, pp. 5652--5657.

\bibitem{DePersis2020}
C.~{De Persis} and P.~{Tesi}, ``Formulas for data-driven control:
  Stabilization, optimality, and robustness,'' \emph{IEEE Transactions on
  Automatic Control}, vol.~65, no.~3, pp. 909--924, March 2020.

\bibitem{vanWaarde2020}
H.~J. {van Waarde}, J.~{Eising}, H.~L. {Trentelman}, and M.~K. {Camlibel},
  ``Data informativity: a new perspective on data-driven analysis and
  control,'' \emph{IEEE Transactions on Automatic Control}, vol.~65, no.~11,
  pp. 4753--4768, 2020.

\bibitem{Berberich2020}
J.~{Berberich}, A.~{Koch}, C.~W. {Scherer}, and F.~{Allg\"ower}, ``Robust
  data-driven state-feedback design,'' in \emph{Proc. of the American Control
  Conference}, 2020, pp. 1532--1538.

\bibitem{Berberich2020d}
J.~{Berberich}, C.~W. Scherer, and F.~{Allg\"ower}, ``Combining prior knowledge
  and data for robust controller design,'' \emph{arxiv.org/abs/2009.05253},
  2020.

\bibitem{Koch2020}
A.~{Koch}, J.~{Berberich}, and F.~{Allg\"ower}, ``Verifying dissipativity
  properties from noise-corrupted input-state data,'' in \emph{Proc. of the
  IEEE Conference on Decision and Control}, 2020, pp. 616--621.

\bibitem{vanWaarde2020d}
H.~J. {van Waarde}, M.~K. {Camlibel}, and M.~{Mesbahi}, ``From noisy data to
  feedback controllers: non-conservative design via a matrix {S}-lemma,''
  \emph{IEEE Transactions on Automatic Control}, pp. 1--1, 2020.

\bibitem{Yakubovich1977}
V.~A. Yakubovich, ``S-procedure in nonlinear control theory,'' \emph{Vestnik
  Leningrad University Mathematics}, vol.~4, pp. 73--93, 1977.

\bibitem{Finsler1936}
P.~Finsler, ``{\"Uber} das vorkommen definiter und semidefiniter formen in
  scharen quadratischer formen,'' \emph{Commentarii Mathematici Helvetici},
  vol.~9, pp. 188--192, 1936.

\bibitem{Zi-zong2010}
Y.~Zi-Zong and G.~Jin-Hai, ``Some equivalent results with {Yakubovich's
  S-Lemma},'' \emph{SIAM Journal on Control and Optimization}, vol.~48, no.~7,
  pp. 4474--4480, 2010.

\bibitem{Willems2005}
J.~C. Willems, P.~Rapisarda, I.~Markovsky, and B.~L.~M. {De Moor}, ``A note on
  persistency of excitation,'' \emph{Systems \& Control Letters}, vol.~54,
  no.~4, pp. 325--329, 2005.

\bibitem{vanWaarde2020c}
H.~J. {van Waarde}, C.~{De Persis}, M.~K. {Camlibel}, and P.~{Tesi}, ``Willems'
  fundamental lemma for state-space systems and its extension to multiple
  datasets,'' \emph{IEEE Control Systems Letters}, vol.~4, no.~3, pp. 602--607,
  2020.

\bibitem{vanWaarde2021}
H.~J. {van Waarde}, ``Beyond persistency of excitation: Online experiment
  design for data-driven modeling and control,''
  \emph{arxiv.org/abs/2102.11193}, 2021.

\bibitem{Boyd1994}
S.~Boyd, L.~E. Ghaoui, E.~Feron, and V.~Balakrishnan, \emph{Linear Matrix
  Inequalities in System and Control Theory}, ser. Studies in Applied
  Mathematics.\hskip 1em plus 0.5em minus 0.4em\relax Society for Industrial
  and Applied Mathematics, 1994.

\end{thebibliography}

\end{document}